\title{Permutation Models Arising From Topological Ideals}
\author{Justin Young\thanks{
                  University of Florida}
        }
\documentclass{article}
\usepackage{graphicx}
\usepackage{amssymb, amsmath}
\usepackage{mathrsfs}
\usepackage{amsthm}
\usepackage{enumerate}
\usepackage{comment}
\usepackage{biblatex}
\newcommand{\beqs}{\begin{equation*}}
\newcommand{\eeqs}{\end{equation*}}
\newtheorem{theorem}{Theorem}[section]

\newtheorem{corollary}[theorem]{Corollary}

\newtheorem{defin}[theorem]{Definition}

\newtheorem{prop}[theorem]{Proposition}

\let\temp\bigg
\let\bigg\Big
\let\Big\temp
\let\frac\dfrac

\newcommand{\pstab}[1]{\text{pstab}(#1)}
\newcommand{\stab}[1]{\text{stab}(#1)}
\newcommand{\cone}[1]{\text{cone}(#1)}

\begin{document}
\makeatletter
\def\imod#1{\allowbreak\mkern10mu({\operator@font mod}\,\,#1)}
\makeatother
\newpage
\maketitle
\begin{abstract}
    A recent paper by Zapletal \cite{dideals2} discusses permutation models of set theory which arise from dynamical ideals and highlights properties of the dynamical ideal which relate to fragments of choice in the permutation model. In this paper, we provide several examples from topology which illustrate using these connections to argue that the corresponding permutation model satisfies either the axiom of countable choice or well-ordered choice.
\end{abstract}




\section{Introduction} In \cite{dideals2}, the notion of dynamical ideals are introduced as a method of constructing permutation models of set theory, and various theorems are proved that relate dynamical properties of the dynamical ideal with fragments of the axiom of choice in the corresponding model, allowing one to determine properties of the model without making explicit reference to the model. In principle, the dynamical properties of a dynamical ideal depend on both the choice of ideal and the underlying space, and in this paper, we look at examples of ideals on natural spaces which satisfy two of those dynamical properties. 

In particular, we show that the ideals generated by closed sets in the following spaces are dynamically \(\sigma\)-complete (definition \ref{Dyn-Sig C.}): \(2^\omega\) (Theorem \ref{CC 2w}), \(\omega^\omega\) (Theorem \ref{CC ww}), \(\mathbb{R}^n\) (Theorem \ref{CC Rn}). We also show that the following ideals support a cone measure (definition \ref{Cone M}): the ideals generated by closed nowhere dense sets on \(2^\omega\) (Theorem \ref{2w-cone}) or any manifold (Theorem \ref{Carpet-cone}), and the ideal generated by compact totally disconnected sets in \(\mathbb{R}^2\) (Theorem \ref{Cone Tot-Disc}). We also show that the ideal generated by compact nowhere dense sets in \(\mathbb{R}^n\) supports a cone measure (Theorem \ref{Compact NWD Cone}), yet give an example of a manifold where the same ideal fails to support a cone measure (Theorem \ref{Compact Non D-Sig C}). This paper leaves open multiple problems, such as the status of the ideal generated by countable closed sets or the ideal generated by closed nowhere dense sets in infinite dimensional spaces.  Additionally, the status of the ideal generated by compact totally disconnected sets in \(\mathbb{R}^n\) for \(n\geq 3\) remains unknown.

This first section introduces the relevant definitions of dynamical ideals, permutation models, and the relevant dynamical properties that we are concerned with. Section 2 investigates examples of dynamical ideals which are dynamically \(\sigma\)-complete (Definition \ref{Dyn-Sig C.}) and Section 3 investigates examples of dynamical ideals which support a cone measure (Definition \ref{Cone M}).
\begin{defin}
    A dynamical ideal is a tuple which consists of a group \(\Gamma\), its action on a set \(X\), and an ideal \(I\) on \(X\) which is invariant under the group action. We denote a dynamical ideal by \((\Gamma\curvearrowright X, I)\).
\end{defin}
This paper will be primarily concerned with dynamical ideals where the underlying set \(X\) is taken to be a topological space and the group \(\Gamma\) is the group of self-homeomorphisms of \(X\) acting by application - unless otherwise stated, the acting group should be taken to be the group of self-homeomorphisms of \(X\). In our discussion of topological spaces, we will use the following notation: given a set \(a\), \(\overline{a}\) denotes the closure, \(bd(a)\) denotes the boundary, and \(int(a)\) denotes the interior of \(a\). Additionally, when a metric is present, \(diam(a)\) denotes the diameter according to the metric, and \(Ball(a,\varepsilon)\) denotes the set of all points \(x\) such that \(d(x,a)<\varepsilon\). If \(x\) is a point, then \(Ball(x,\varepsilon)=Ball(\{x\},\varepsilon).\)
\begin{defin}
    Given a model of set theory with atoms \(V[[X]]\) where \(X\) is the set of atoms, a group action \(\Gamma\curvearrowright X\) extends to a group action \(\Gamma\curvearrowright V[[X]]\) in a natural way: for \(A\in V[[X]]\) such that \(A\notin X\) and \(\gamma\in \Gamma\), let \(\gamma\cdot A=\{\gamma\cdot a: a\in A\}\). We let \(pstab(A)=\{\gamma\in \Gamma: \text{ for all }a\in A, \gamma\cdot a=a\}\) denote the pointwise stabilizer, and \(stab(A)=\{\gamma\in \Gamma: \gamma\cdot A=A\}\) denote the stabilizer.
\end{defin}
\begin{defin}
    Given a dynamical ideal \((\Gamma \curvearrowright X, I)\) and a model of set theory with atoms using \(X\) as the set of atoms \(V[[X]]\), we define the permutation model corresponding to the dynamical ideal by as the transitive part of \(\{A\in V[[X]]: \text{ there exists \(b\in I\) such that \(pstab(b)\subseteq stab(A)\)}\}\). The permutation model is denoted by \(W[[X]].\)
\end{defin}

\begin{defin}\label{Dyn-Sig C.}
    We say the dynamical ideal \((\Gamma\curvearrowright X, I)\) is dynamically \(\sigma\)-complete if for all \(a\in I\) and any countable sequence \((b_i:i\in \omega)\) of sets in \(I\) there exist group elements \(\gamma_i\in pstab(a)\) such that \(\bigcup_{i=0}^\omega \gamma_i\cdot b_i \in I\).
\end{defin}
Theorem 4.3 of \cite{dideals2} states that if a dynamical ideal is dynamically \(\sigma\)-complete, then the associated permutation model satisfies the axiom of countable choice.
\begin{defin}
    We say the dynamical ideal \((\Gamma\curvearrowright X, I)\) has cofinal orbits if for every \(a\in I\) there exists \(b\in I\) which is \(a\)-large: for any \(c\in I\) there exists \(\gamma\in\pstab{a}\) such that \(c\subseteq \gamma\cdot b\).
\end{defin}
Theorem 3.3 of \cite{dideals2} states that if a dynamical ideal has cofinal orbits, then the associated permutation model satisfies the axiom of well-ordered choice. The paper goes on to identify a stronger condition visible in the permutation model which implies that the ideal has cofinal orbits:
\begin{defin}\label{Cone M}
    A dynamical ideal \((\Gamma\curvearrowright X, I)\) supports a cone measure if in the associated permutation model \(W[[X]]\), there is a set \(C\subseteq I\) which is cofinal with respect to inclusion and such that for every set \(D\subseteq C\) there is \(a\in I\) such that the set \(\cone a:=\{b\in C:a\subseteq b\}\) is either a subset of \(D\) or disjoint from \(D\). The set \(C\) is referred to as a support of the cone measure.
\end{defin}

\section{Dynamically \(\sigma\)-Complete Ideals}
In this section, we present examples of dynamically \(\sigma\)-complete ideals. We first make the following definitions: \begin{defin}
    We say a topological space together with an ideal \((X,\tau,I)\) is tight if for any \(a\in I\), \(a\) can be expanded to a set \(a'\in I\) such that for any \(b\in I\) and any open neighborhood \(U\in\tau\) of \(a'\), there exists \(\gamma\in pstab(a')\) with \(\gamma\cdot b\subseteq U\). In the case that \((X,\tau)\) is metrizable, it is equivalent to fix a compatible metric and just consider neighborhoods \(U\) which are open metric balls about \(a'\).
\end{defin} 
\begin{defin}
    Let \((X,\tau)\) be a topological space. An ideal \(I\) is called a \(\sigma\)-ideal of closed sets if it has the form \(I=\{A\subseteq X:\overline{A}\in J\}\) for some \(\sigma\)-ideal \(J\). If \(I\) instead has the form \(I=\{A\subseteq X: \overline{A}\in J \text{ is compact}\}\) then we say \(I\) is a \(\sigma\)-ideal of compact sets.
\end{defin}
When \(I\) is a \(\sigma\)-ideal of closed sets, tightness of a sequential and perfectly normal Hausdorff space is enough to guarantee that the ideal is dynamically \(\sigma\)-complete, as established in the following proposition:
\begin{prop}
    Let \((X,\tau)\) be a sequential and perfectly normal Hausdorff space, let \(\Gamma\) be the group of self-homeomorphisms acting by application, and let \(I\) be a \(\sigma\)-ideal of closed sets. Then if \((X,\tau, I)\) is tight, then the dynamical ideal \((\Gamma\curvearrowright X, I)\) is dynamically \(\sigma\)-complete. In particular, the result holds if \(X\) is metrizable. Further, if \(X\) locally compact, then the result holds even if \(I\) is a \(\sigma\)-ideal of compact sets.
\end{prop}
\begin{proof}
    Let \(J\) be a \(\sigma\)-ideal such that \(I=\{a\subseteq X:\overline{a}\in J\}\). Let \(a\in I\) and \((b_i:i\in \omega)\subseteq I\) be given - without loss of generality assume each of these sets are closed. Enlarge \(a\) to \(a'\in I\) as needed to witness that \((X,\tau,I)\) is tight, and let \((U_i:i\in\omega)\) be a sequence of open neighborhoods of \(a'\) such that \(U_i\supseteq \overline{U_{i+1}}\) and \(\bigcap_{i\in\omega}U_i=a'\) witnessing that \((X,\tau)\) is perfectly normal. Let \(\gamma_i\in\pstab {a'}\) be such that \(\gamma_i\cdot (b_i)\subseteq U_i\). 
    
    Letting \(c=a'\cup \bigcup_{i\in \omega} \gamma_i\cdot b_i\), we note that since \(a'\) and each of the \(\gamma_i\cdot b_i\) belong to \(J\), it follows that \(c\in J\). To conclude that \(c\in I\), it remains to show that \(c\) is closed. To see this, let \(x\) belong to the closure of \(c\) and use the fact that the space is sequential to find a sequence \((x_i:i\in \omega)\) in \(c\) which converges to \(x\). If there is some \(\gamma_k\cdot b_k\) which the sequence \((x_i)\) visits infinitely often, then we can pass to a subsequence and use the fact that \(\gamma_k\cdot b_k\) is closed to conclude \(x\in \gamma_k\cdot b_k\). If not, then \((x_i:i\in\omega)\) visits each \(\gamma_k\cdot b_k\) finitely many times and we can pass to a subsequence \((x_{i_\ell}:\ell\in\omega)\) such that if \(x_{i_\ell}\in b_k\), then \(\ell\leq k\). We show that \(x\in a'\) by showing that \(x\in U_k\) for each \(k\): let \(k\in \omega\) be given, and suppose towards a contradiction that \(x\not\in U_k\). Then \(x\not\in \overline{U_{k+1}}\), and we can find an open neighborhood \(V\ni x\) which separates \(x\) from \(U_{k+1}\). But note that by assumption, \(U_{k+1}\) and \(V\) both contain all but finitely many elements from the sequence \((x_{i_\ell}:\ell\in\omega)\), a contradiction. 

    When \(X\) is locally compact and we consider the ideal generated by countable compact sets, we modify the above argument by choosing neighborhoods \(U_n\) of \(a'\) such that \(\overline{U_n}\) is compact. Hence in the end, \(c\subseteq U_0\) is also compact.
\end{proof}

We can use the above result to show that certain topological spaces yield dynamically \(\sigma\)-complete ideals. The next few theorems illustrate spaces which are tight, hence correspond to a dynamical ideal which is dynamically \(\sigma\)-complete.

\begin{theorem}\label{CC 2w}
    The Cantor space \(X=2^\omega\) with the usual topology and the ideal generated by countable closed subsets is tight.
\end{theorem}
\begin{proof}
    Let \(a,b\in I\) be two closed sets, and let \(U\) be an open neighborhood of \(a\). By compactness of \(a\), we can find a clopen set \(V\subseteq U\) which contains \(a\); write \(b'=b\setminus U\) to obtain a closed set which is disjoint from \(V\). Let \(\mathcal{C}\) be a cover of \(b'\) by pairwise disjoint neighborhoods of the form \(N_s=\{x:x(n)=s(n) \text{ for all } n<|s|\}\) for a finite string \(s\in 2^{<\omega}\) - by compactness, \(\mathcal{C}=\{N_{s_0},\dots,N_{s_k}\}\) can be taken to be finite. Because \(a\cup b\) is closed and countable, \(a\cup b\) is nowhere dense, so let \(\mathcal{O}\subseteq V\) be an open open set disjoint from \(a\cup b\) with the form \(\mathcal{O}=N_t\) for some finite string \(t\in 2^{<\omega}\). Let \(V_i=N_{t_i}\) where \(t_i=t^\frown 1^\frown\dots^\frown 1^\frown 0\) is the string obtained from \(t\) by concatenating \(i\) 1s and then a final 0. We define a homeomorphism \(\varphi_i:(N_{s_i}\cup V_i)\rightarrow (N_{s_i}\cup V_i)\) for \(i\leq k\) by \(\varphi_i(s_i^\frown y)=t_i^\frown y\) and \(\varphi_i(t_i^\frown y)=s_i^\frown y\). In words, the map \(\varphi_i\) swaps the sub-tree of \(2^\omega\) which starts at \(s_i\) with the subtree which starts at \(t_i\). The union of these homeomorphisms, together with the identity elsewhere on \(2^\omega\) defines a homeomorphism of the entire Cantor space. By construction, the homeomorphism fixes \(a\) pointwise and is such that the image of \(b\) is contained in \(V\).
    \end{proof}
We obtain a similar result for the Baire space \(\omega^\omega\), but must first establish a preliminary proposition:
\begin{prop}\label{clopen shrink}
    Let \(a\subseteq \omega^\omega\) be closed countable, and \(U\subseteq\omega^\omega\) an open subset containing \(a\). There is a clopen subset \(V\subseteq U\) which contains \(a.\)
\end{prop}
\begin{proof}
    We prove by induction on the Cantor Bendixson rank of \(a\). Let \(a,U\) be given as in the statement of the proposition, and let \(\alpha\) be the least ordinal such that \(a^{(\alpha)}\neq \emptyset\) and \(a^{(\alpha+)}=\emptyset.\) Then \(a^{\alpha}\) contains no limit points; let \(b=a^{(\alpha)}\) and write  \(b=\{x_i:i\in\omega\}\), fix an ultrametric \(d\) on \(\omega^\omega\) and recursively construct sets \(V_i\subseteq U\) such that \begin{enumerate}
        \item \(V_j\) is clopen
        \item for \(j<i\), \(V_i\cap V_j=\emptyset\)
        \item \(diam(V_i)<2^{-i}\)
        \item If \(x_i\in V_j\) for some \(j<i\), then \(V_i=\emptyset\), and otherwise \(V_i\ni x_i\).
    \end{enumerate}
    We note that since \(U\setminus \bigcup_{j<i} V_j\) is open, it is easy to construct \(V_i\) in this manner by choosing it to be a neighborhood of \(x_i\). In the end, \(V=\bigcup_{i\in\omega}V_i\) is open; to see \(V\) is closed, suppose \(y\in \overline{V}\), and let \((y_m:m\in\omega)\) be a sequence of elements from \(V\) which converges to \(y\). Each \(y_m\) belongs to a unique \(V_i\), denote the particular set by \(V_{i_m}\), and observe that this induces a sequence of elements of \(b\), namely \((x_{i_m}:m\in\omega)\). Use the fact that \(d\) is an ultrametric to write \(d(x_{i_m},y)\leq \max\{d(x_{i_m},y_m),d(y_m,y)\}\). If \(y=x_{i_m}\) for some \(m\), then \(y\in V\). Otherwise, by the assumption that \(b\) contains no accumulation points, we observe that \(d(x_{i_m},y)\) is bounded away from \(0\), and hence so is \(d(x_{i_m},y_m)\). In particular, the sequence \(diam(V_{i_m})\) is bounded away from \(0\), so we conclude the sequence visits only finitely many \(V_i\), and since the union of finitely many \(V_i\) is closed, we see that \(y\in V\).

    Now observe that \(a^*=a\setminus V\) is closed and has Cantor Bendixson rank \(\leq \alpha\). By induction, there is a clopen set \(V^*\subseteq U\setminus V\) which contains \(a^*\), and hence \(V\cup V^*\) is clopen and contains \(a\).
\end{proof}
\begin{theorem}\label{CC ww}
    The Baire space \(X=\omega^\omega\) with the usual topology and the ideal generated by countable closed subsets is tight.
\end{theorem}
\begin{proof}
    Let \(a,b\in I\) be two closed sets, and let \(U\) be an open neighborhood of \(a\). Use Proposition \ref{clopen shrink} to find a clopen set \(V\subseteq U\) such that \(a\subseteq V\). Use the fact that since \(a\cup b\) is closed and countable, the union is nowhere dense to find a clopen \(\mathcal{O}\subseteq V\) which is disjoint from \(a\cup b\). Next we use the fact that any nonempty clopen subset of \(\omega^\omega\) is homeomorphic to \(\omega^\omega\) to find a homeomorphism \(\varphi_1:\mathcal{O}\rightarrow \omega^\omega\setminus V\). Observe that the union of \(\varphi_1\) together with its inverse and the identity on \(V\setminus \mathcal{O}\) yields a homeomorphism of \(\omega^\omega\) which fixes \(a\) pointwise and maps \(b\) into \(V\).
\end{proof}

In \cite{dideals2}, it is shown that the ideal generated by countable compact sets in \(\mathbb{R}\) is tight. This result is true when generalized to \(\mathbb{R}^n\), but requires a different proof. To build up to the proof, we first consider the following propositions:

    \begin{prop}\label{shrink}
        Let \(A,B\subseteq \mathbb{R}^n\) be countable compact, and let \(U\subseteq X\) be a metric open ball containing \(A\). Then there is a metric open ball with rational radius \(V\) with the same center as \(U\) such that \(A\subseteq V\subseteq \overline{V}\subseteq U\), and such that the boundary of \(V\) is disjoint from \(B\).
    \end{prop}
    \begin{proof}
        Let \(U=Ball(x,r)\), and define a continuous function \(f:B\rightarrow\mathbb{R}\) by \(f(y)=d(x,y)\). Since \(B\) is countable compact, so is the image of \(f\), hence \(f(B)\) is nowhere dense. Let \(\delta=\inf\{\varepsilon: A\subseteq Ball(x,\varepsilon)\}\), and note that \(\{s\in\mathbb{Q}:\delta<s<r \text{ and } s\not\in f(B)\}\) is nonempty. Fix some \(s\) from this set, and note that \(V=Ball(x,s)\) works.
    \end{proof}
    \begin{prop}\label{disjointify}
        Let \(a\subseteq \mathbb{R}^n\) be countable compact, and let \(\varepsilon>0\). There is a finite cover \(\mathcal{C}\) of \(a\) consisting of pairwise disjoint metric open balls of radius less than \(\varepsilon\). In particular, \(\bigcup \mathcal{C}\subseteq Ball(a,\varepsilon)\).
    \end{prop}
    \begin{proof}
        We proceed by induction on the Cantor-Bendixson rank of \(a\). Let \(\alpha\) be the least ordinal such that \(a^{(\alpha)}\neq \emptyset\) and \(a^{(\alpha+)}=\emptyset\). Then \(a^{(\alpha)}\) contains no accumulation points; we write \(a^{(\alpha)}=\{x_i:i\in\omega\}\), and we recursively construct balls \(B_i=Ball(x_i,\varepsilon_i)\) with \(\varepsilon_i< \varepsilon\) such that \begin{enumerate}            \item for \(j<i\), \(B_i\cap B_j =\emptyset\)
            \item for all \(x\in a\) and all \(i\in \omega\), \(x\not\in bd(B_i)\)
            \item If \(x_i\in B_j\) for some \(j<i\), then \(B_i=\emptyset\), and otherwise \(B_i\ni x_i\).
        \end{enumerate}
        We note that property 2 can be ensured by the proposition above, and properties 1 and 3 can be ensured since \(a^{(\alpha)}\) has no accumulation points. Further, since \(a^{\alpha}\) is compact, we note that the process will yield finitely many sets \(B_i\) - let \(k\) denote the index of the last nonempty set constructed. In the end, we let \(a^*=a\setminus \bigcup_{i\leq k}B_i\), and we note \(a^*\) is closed of rank \(\leq \alpha\). We let \(\delta=\inf\{d(x,y):x\in a^*, y\in\bigcup_{i\leq k}B_i\}\) - we note \(\delta>0\) since \(B_i\) were chosen to not intersect \(a\) in its boundary. Let \(\delta'=\min\{\delta,\varepsilon\}\), and by induction, we can cover \(a^*\) with pairwise disjoint open balls of radius smaller than \(\delta\) and which are also disjoint from the \(B_i\). Taking the cover of \(a^*\) together with the \(B_i\) gives a cover of \(a\) as desired.
    \end{proof}
    \begin{prop}
        Let \(a\subseteq \mathbb{R}^n\) be countable compact, and let \(\mathcal{C}\) be a cover of \(a\). There exists a refinement of \(\mathcal{C}\) into a cover of \(a\) by pairwise disjoint metric balls with rational radius.
    \end{prop}
    \begin{proof}
        Let \(a\subseteq X\) be countable compact, and let \(\mathcal{C}\) be a cover of \(a\). Use Proposition \ref{shrink} to replace \(\mathcal{C}\) with a refinement \(\mathcal{C}_1\) consisting of balls whose boundary do not intersect \(a\). Let \(\varepsilon=\min\{d(x,y):x\in a, y\in bd(C)\text{ for }C\in\mathcal{C}_1\}\). Now use Proposition \ref{disjointify} to construct a cover \(\mathcal{C}_2\) of \(a\) by pairwise disjoint balls of radius smaller than \(\varepsilon\); the choice of \(\varepsilon\) ensures that \(\mathcal{C}_2\) is a refinement of \(\mathcal{C}_1\).
    \end{proof}
    \begin{prop}\label{cylinder}
        Let \(\{C_1,\dots, C_m\}\) be a collection of pairwise disjoint closed metric balls in \(\mathbb{R}^n\). There is a set \(K\) which is the image of the closed unit ball under some homeomorphism of \(\mathbb{R}^n\) and satisfies \(C_1,C_2\subseteq K\) and for \(3\leq i\leq m\), \(C_i\cap K=\emptyset\).
    \end{prop}
    \begin{proof}
        Let \(L\) be a straight line path from the center of \(C_1\) to the center of \(C_2\). Since the balls are all closed and pairwise disjoint, we can find larger balls \(D_3\supseteq C_3, D_4\supseteq C_4,\dots, D_m\supseteq C_m\) such that \(C_1,c_2,D_3,\dots, D_m\) are all pairwise disjoint. If \(L\) passes through \(D_i\), replace \(L\) by a curve that it traces the geodesic along the boundary of \(D_i\). By this construction, \(L\) does not have a knot, so there is a homeomorphism \(\psi\) of \(\mathbb{R}^n\) that sends \(L\) to the first axis; i.e. \(\psi(L)\subseteq \{(x,0,\dots,0):x\in\mathbb{R}\). Construct a family of cylinders about \(L\) as follows: \(cyl_\rho L:=\psi^{-1}(\{(x,r_2,\dots,r_{n}):\sqrt{r_2^2+\dots+r_n^2}\leq \rho, (x,0,\dots, 0)\in \psi(L)\})\). Since there are finitely many \(C_m\), we can choose \(\rho\) such that \(cyl_\rho L\) does not intersect any of the \(C_3,\dots,C_m\). In the end, let \(K=C_1\cup cyl_\rho L\cup C_2\). 
    \end{proof}
We are now finally ready to prove that \(\mathbb{R}^n\) with the ideal generated by countable compact sets is tight.
\begin{theorem}\label{CC Rn}
    Euclidean space \(\mathbb{R}^n\) with the usual topology and the ideal generated by compact countable sets is tight. Moreover, tightness can be witnessed by a homeomorphism with compact support.
\end{theorem}
\begin{proof}
     Let \(a,b\) be countable compact, and let \(\varepsilon>0\). Find a finite cover of \(a\), denoted by \(\mathcal{C}=\{C_1,\dots,C_m\}\) of pairwise disjoint basic open balls of radius at most \(\varepsilon/2\) whose boundaries are disjoint from \(b\). Here we note that if \(x\in b\cap\bigcup\mathcal{C}\), then \(x\) is already close enough to \(a\) to satisfy the tightness condition. Hence we let \(a'=a\cup (b\cap\bigcup\mathcal{C})\) and \(b'=b\setminus a'\) and observe that \(a'\) and \(b'\) are still compact and countable. Now choose \(\delta\) small enough such that \(Ball(b',\delta)\) is disjoint from \(\bigcup\mathcal{C}\), and find a cover \(\mathcal{D}=\{D_1,\dots,D_\ell\}\) of \(b'\) by pairwise disjoint balls of radius smaller than \(\delta\). 
    
    We use induction on \(\ell\) to constuct a homeomorphism \(\gamma\) of \(\mathbb{R}^n\), such that \(\gamma(b')\subseteq \bigcup\mathcal{C}\) and \(\gamma\in\pstab a'\). The base case \(\ell=0\) is trivial. For general \(\ell\), we start by finding \(\gamma_1\) such that \(\gamma_1(b'\cap D_1)\subseteq \bigcup \mathcal{C}\) and \(\gamma_1\) fixes all other points of \(a'\) and \(b'\). To do this, first use Proposition \ref{shrink} to find a ball \(C_1'\subseteq \overline{C_1'}\subseteq C_1\) which contains \(a'\cap C_1\). Let \(D_1^*\subseteq C_1\setminus C_1'\)  be a ball whose boundary is disjoint from the boundary of \(C_1\setminus C_1'\); \(D_1^*\) is the target where elements of \(b'\cap D_1\) will be sent. Use Proposition \ref{cylinder} to find a set \(K\) which contains \(D_1,D_1^*\), and is disjoint from \(C_1',C_2,\dots, C_m, D_2,\dots, D_\ell.\) Let \(f:K\rightarrow K\) be a homeomorphism of \(K\) which fixes the boundary and is such that \(f(b\cap D_1)\subseteq D_1^*.\) Let \(\varphi_1:\mathbb{R}^n\rightarrow\mathbb{R}^n\) be the homeomorphism obtained by extending \(f\) with the identity outside of \(K\), observing that \(\varphi_1\) fixes \(a'\) pointwise and moves \(b'\cap D_1\) within \(\varepsilon\) of \(a\). 
    
    Since \(a'\cup \varphi_1(b'\cap D_1)\) is still countable compact, we apply the induction hypothesis to obtain a \(\varphi_2\) which fixes \(a'\cup \varphi_1(b'\cap D_1)\) pointwise and moves \(b'\setminus D_1\) into \(\bigcup\mathcal{C}\). Taking the composition of the two homeomorphisms, we see that \(\varphi=\varphi_2\varphi_1\) is the desired homeomorphism to witness that the space is tight. Further, we observe that \(\varphi\) is a homeomorphism with compact support, so the result holds even if the acting group is restricted to only homeomorphisms with compact support.
    \end{proof}

\begin{theorem}
    Euclidean space \(\mathbb{R}^n\) with the usual topology and the ideal generated by closed countable sets is tight.
\end{theorem}
\begin{proof}
    Let \(a\) be given countable closed, and let \(a'\) extend \(a\) by adding in every point where each coordinate has a fractional part of \(1/2\). Let \(b\) be a given countable closed set, and \(\varepsilon>0\). Now consider hyperplanes of the form \(x_i=m\), where \(1\leq i\leq n\) and \(m\in \mathbb{Z}\). Some of these hyperplanes may intersect \(a'\cup b\), but because \(a'\cup b\) is countable, there exists some real number \(r_{i,m}\in(m-10^{-1},m+10^{-1})\) such that \(x_i=r_{i,m}\) is disjoint from \(a'\cup b\). Note that the collection of all such hyperplanes divide the \(\mathbb{R}^n\) into regions which are homeomorphic to the cube \([0,1]^n\), and each of which has nontrivial intersection with \(a'\) (in particular, each region will have exactly one point where each coordinate has a fractional part of \(1/2\)). Enumerate these regions by \(U_j\), and let \(a_j'=a'\cap U_j, b_j=b\cap U_j\). By construction, \(a_j', b_j\) are compact in \(int(U_j)\), and so by the previous result, there is a homeomorphism \(\varphi_j\) of \(U_j\) that fixes its boundary pointwise, fixes \(a_j'\) pointwise, and brings \(b_j\) within \(\varepsilon\) of \(a_j'\). The union of all the \(\varphi_j\) is well-defined because they are all the identity on the boundaries of their respective regions, and taking the union yields a homeomorphism of the entire space \(\mathbb{R}^n\) which fixes \(a'\) pointwise and moves \(b\) within \(\varepsilon\) of \(a'\).
\end{proof}

We finish this section with examples of ideals which are not dynamically \(\sigma\)-complete. This first class of ideals differs from what has been discussed above in that we fix a metric and restrict the acting group to the group of self-isometries instead of the group of all self-homeomorphisms. It turns out that restricting the group in this way is enough to prevent the ideal from being dynamically \(\sigma\)-complete.
\begin{theorem}
    Let \((X,d)\) be an uncountable separable metric space, and let \(I\) contain all countable closed sets and exclude \(X\) itself. Then \((Iso(X)\curvearrowright X, I)\) is not dynamically \(\sigma\)-complete.
\end{theorem}
\begin{proof}
    For a counterexample, let \(a=\emptyset\) and let \(b_n\) be a \(1/n\)-net of \(X\). That is, for each \(x\in X\), \(d(x,b_n)<1/n\). Then regardless of the choice of group elements \(\gamma_n\in Iso(X)\), we see that \(\gamma_n\cdot b_n\) is still a \(1/n\)-net of \(X\). In the end, \(\bigcup_{n\in \omega}\gamma_n\cdot b_n\) is a dense subset of \(X\) and hence its closure is \(X\) and does not belong to the ideal.
\end{proof}

We introduce this next non-example of dynamical \(\sigma\)-completeness in anticipation of how it contrasts with examples given in the following section. We first make the following definition:
\begin{defin}
    We define the Cantor Set Iteration to be a set \(X\subseteq\mathbb{R}^2\) as follows: let \(C_0=[0,-1]\) and let \(C_{i+1}=\frac{1}{3}C_i\cup \left(\frac{2}{3}+\frac{1}{3}C_i\right)\). Let \(K_i=\bigcup_{j\leq i} C_j\times [j,j+1]\). In the end, let \(X=\bigcup_{i\in\omega} K_i\).
\end{defin}

\begin{theorem}\label{Compact Non D-Sig C}
    Let \(X\) be the Cantor Set Iteration as defined above. Let \(\Gamma\) denote the group of self-homeomorphisms of \(X\) acting by application, and let \(I\) denote the ideal generated by compact nowhere dense sets. The ideal \((\Gamma\curvearrowright X, I)\) does not have cofinal orbits.
\end{theorem}
\begin{proof}
    Before showing that the ideal is not dynamically \(\sigma\)-complete, we first make a few remarks: let \(K_i\subseteq X\) be as defined in the construction of \(X\) and observe that \(X\setminus K\) consists of \(2^i\) components, each of which is has non-compact closure. Observe that any compact set \(a\subseteq X\) is contained in \(K_i\) for some \(i\), and that if \(a\subseteq b\) are both compact sets, then the number components of \(X\setminus a\) with non-compact closure is less than or equal to the number such components of \(X\setminus b\): if \(U\) is a component of \(X\setminus a\) such that \(\overline{U}\) is not compact, then note \(b\cap\overline{U}\) is compact. Hence \(\overline{U}\setminus b\) consists of components, at least one of which has non-compact closure.

    Now we show that the ideal does not have cofinal orbits: we let \(a=\emptyset\), and choose \(b_i\) to be compact nowhere dense such that \(X\setminus b_i\) contains \(2^i\) components whose closure is not compact (in particular, one can take \(b_i\) to be the boundary of \(K_i\)). Let \(h_i:X\rightarrow X\) be any homeomorphism; observe that  \(X\setminus h_i(b_i)\) still contains \(2^i\) components whose closure is not compact. Then we consider the set \(B=\bigcup_{n\in\omega} h_i(b_i)\) and show it has non-compact closure. Suppose to the contrary that \(c\supseteq B\) is a compact set, then \(c\subseteq K_i\) for some \(i\). Hence \(X\setminus c\) must contain at most \(2^i\) components which have non-compact closure. However, \(h_{i+1}(b_{i+1})\subseteq c\), so \(c\) must have at least \(2^{i+1}\) components which have non-compact closure, and we have reached a contradiction.
\end{proof}

\section{Ideals with cofinal orbits} In this section we present a couple examples of ideals which support a cone measure and hence have cofinal orbits. Our first examples come from considering the ideal of nowhere dense subsets. We are aware that while working on similar problems, \cite{Anett} has independently shown that the nowhere dense ideal has cofinal orbits in \(\mathbb{R}^n\). Here we show that the nowhere dense ideal supports a cone measure in any manifold, where an \(n\)-dimensional manifold is a second countable Hausdorff space which is locally Euclidean: every point has a neighborhood which is homeomorphic to either \(\mathbb{R}^n\) or \(\mathbb{R}_{\geq 0}\times \mathbb{R}^{n-1}.\) We make no distinction between whether or not a manifold has a boundary, as the following definition and all theorems apply equally to both. We first make the definition:
\begin{defin}
    Let \(X\) be an \(n\)-dimensional connected manifold (with or without boundary), and let \(A\subseteq X\). Say that \(A\) is a flat S-carpet in \(X\) if \begin{enumerate}
        \item \(A\) is connected, closed, and nowhere dense.
        \item There are countably many components \(U_i\) of \(X\setminus A\) and each component is open.
        \item The set of closures of the complementary components \(\{\overline{U_i}:i\in \omega\}\) is pairwise disjoint.
        \item For each component \(U_i\) of \(X\setminus A\), there is an open set \(V_i\supsetneq \overline{U_i}\) such that \(h_i:V_i\rightarrow \mathbb{R}^n\) is a homeomorphism and also induces a homeomorphism between \(U_i\) and the open unit ball in \(\mathbb{R}^n\).
        \item If \((x_i:i\in\omega)\) is a sequence such that \(x_i\in \overline{U_i}\) and some subsequence converges \(x_{i_\ell}\rightarrow x\), then \(\overline{U_{i_\ell}}\rightarrow \{x\}\) in the hyperspace.
    \end{enumerate}
\end{defin}
We note that if \(X\) is a manifold with boundary and \(A\) is a flat S-carpet in \(X\), then \(A\) necessarily contains all boundary points of \(X\) by condition 4 in the definition. We also note that when \(X\) is compact and has a metric, condition 5 is equivalent to saying that \(diam(U_i)\rightarrow 0\).

The authors of \cite{Whyburn} and \cite{Cannon} study flat S-carpets in the sphere \(S^n\) (using the metric characterization of condition 5 instead of the one given here). We summarize the relevant results of the papers in the following theorem, which will lead to two immediate corollaries:
\begin{theorem} 
    Let \(X, Y\) be flat S-carpets in the sphere \(S^n\). Any homeomorphism between \(X,Y\) will induce a correspondence between the complementary components of the two sets. Further, any partial correspondence between the sets of complementary components can be extended to a homeomorphism between \(X,Y\). That is, if \(U_0,\dots, U_k, V_0,\dots, V_k\) are finitely many complementary components of \(X, Y\) respectively, then given homeomorphisms \(\varphi_i:bd(U_i)\rightarrow bd(V_i)\) between the boundaries of the complementary components, one can obtain a homeomorphism from \(\overline{\varphi}:X\rightarrow Y\) which extends each \(\varphi_i\).
\end{theorem}
Of note is that \cite{Cannon} only claims the results for \(n\neq 4\). However, developments have occurred since the paper was first published, namely that the annulus theorem since been proven in the relevant case, so we observe the results hold in the case \(n=4\) here. We get the following immediate corollaries.

\begin{corollary}
    Let \(X, Y\) be two flat S-carpets in \(S^n\), and let \(h:X\rightarrow Y\) be a homeomorphism between the two carpets. Then \(h\) can be extended to a homeomorphism \(\overline{h}:S^n\rightarrow S^n\).
\end{corollary}
\begin{proof}
    The previous theorem tells us that \(h\) induces a correspondence between the complementary components of \(X,Y\). Let \(U_i\) denote the complementary components of \(X\) and denote the corresponding complementary components of \(Y\) induced by \(h\) as \(V_i\). Let \(h_i:\overline{U_i}\rightarrow\overline{V_i}\) be a homeomorphism extending \(h\restriction bd(U_i)\). Let \(\overline{h}=h\cup\bigcup_{i\in\omega}h_i\) - this is well-defined since \(h_i\) equals \(h\) everywhere the domains overlap. 
    
    To see that \(\overline{h}\) is continuous, let \((x_m:m\in\omega)\) be a sequence such that \(x_m\rightarrow x\). If either infinitely many \(x_m\) belong to \(X\) or infinitely many \(x_m\) belong to the same \(U_i\), then \(\overline{h}(x_m)\rightarrow \overline{h}(x)\) by continuity of \(\overline{h}\) on \(X\) or \(U_i\) respectively. Otherwise we can pass to a subsequence such that \(x_m\) each belong to distinct \(U_{i_m}\). In this case, we use the fact that \(X\) is a flat S-carpet to note that it satisfies condition 5 in the definition above and find a sequence consisting of \(x'_m\in bd(U_{i_m})\) such that \(x'_m\rightarrow x\). Since \(x'_m\in X\), we see that \(\overline{h}(x'_m)\rightarrow \overline{h}(x)\). From here, using the fact that \(Y\) satisfies condition 5 tells us that \(\overline{h}(x_m)\rightarrow \overline{h}(x)\). Hence \(\overline{h}\) is continuous, and the same argument shows that \(h^{-1}\) is continuous.
\end{proof}
\begin{corollary}\label{Filled-homeomorphism}
    Let \(M\) be a manifold obtained by removing \(k\) disjoint flat balls from the sphere \(S^n\), and let \(X,Y\subseteq M\) be flat S-carpets in \(M\). There is a homeomorphism \(\varphi\) of \(M\) which fixes the boundary of \(M\) pointwise and whose restriction to \(X\) is a homeomorphism from \(X\) to \(Y\). 
    In particular, if \(k=1\), then we see that \(M\simeq [0,1]^n\) is the \(n\)-dimensional cube, and if \(k=2,\) then \(M\simeq [0,1]\times S^{n-1}\) is the annulus.
\end{corollary}
\begin{proof}
    Let \(\varphi\) embed \(M\) into \(S^{n}\) and let \(h:\varphi(X)\rightarrow \varphi(Y)\) be a homeomorphism which extends the identity map on the boundary of \(\varphi(M)\). Extend \(h\) to \(\overline{h}\), a homeomorphism of the entire sphere as in the previous corollary. Observe that \(h(\varphi(M))\) is a homeomorphism of \(\varphi(M)\), which fixes the boundary and hence the composition \(\varphi^{-1}h\varphi(M)\) is a homeomorphism of \(M\) which maps \(X\) to \(Y\) and fixes the boundary of \(M\) pointwise.
    \end{proof}
We will show that in any manifold \(X\), the set of flat S-carpets in \(X\) is the support of a cone measure in the permutation model. We first show that this set is cofinal:
\begin{prop}\label{superset-Carpet}
    Let \(X\) be a separable \(n\)-dimensional manifold (with or without boundary), let  \(a\subseteq X\) be closed nowhere dense. There is a flat S-Carpet \(K\) in \(X\) such that \(a\subseteq K\).
\end{prop}
\begin{proof}
    Let \((W_i:i\in \omega)\) be a countable basis for \(X\). Recursively construct open sets \(V_i,U_i\) as follows: if \(W_i\subseteq U_j\) for some \(j<i\), then let \(V_i,U_i=\emptyset\). Otherwise, let \(V_i\subseteq W_i\) be open such that for \(j<i\), either \(V_i\cap W_j=\emptyset\) or \(V_i\subseteq W_j\). Further, choose \(V_i\) to be disjoint from \(a\) and \(U_j\) for \(j<i\), and let \(\varphi_i:\mathbb{R}^n\rightarrow V_i\) be a homeomorphism. Let \(B\) denote the open unit ball in \(\mathbb{R}^n\), and let \(U_i\subseteq V_i\) be the open set given by \(\varphi_i(B)\). Let \(K=X\setminus\bigcup_{i\in\omega}U_i\). By construction, \(K\) satisfies requirements 1-4 to be a flat S-carpet in \(X\). 
    
    To see that requirement 5 is satisfied, suppose \((x_i:i\in\omega)\) is a sequence such that \(x_i\in \overline{U_i}\), and \(x_i\) has a subsequence \((x_{i_\ell}:\ell\in\omega)\) which converges to the point \(x\). Let \(\mathcal{O}=W_k\) be a basic open neighborhood of \(x\). Now let \(N>k\) be such that if \(\ell\geq N\), then \(x_{i_\ell}\in \mathcal{O}\). Observe that since \(x_{i_\ell}\in \overline{U_{i_\ell}}\), we know \(\overline{U_{i_\ell}}\) has nonempty intersection with \(\mathcal{O}\) and is hence contained in \(\mathcal{O}\).
\end{proof}

\begin{theorem} \label{Carpet-cone}
    Let \(X\) be a \(n\)-dimensional connected manifold (with or without boundary), and let \(I\) be the ideal generated by closed nowhere dense sets. Let \(C\subseteq I\) be the set of flat S-carpets in \(X\). \(C\) is the support of a cone measure.
\end{theorem}
\begin{proof}
    Working in the permutation model, let \(D\subseteq C\) be a cofinal subset, and let \(a\in I\) be such that \(\pstab a\subseteq \stab D\). Without loss of generality, we can assume that \(a\in C\) by enlarging via the previous proposition. Let \(b\in C\) be such that for each complementary component \(U_i\) of \(a\), \(b\cap \overline{U_i}\) is a flat S-carpet in \(\overline{U_i}\). We will show that \(\cone b\subseteq D\): let \(c\in C, d\in D\) such that \(b\subseteq c,d\). Let \(c_i=c\cap U_i\), and \(d_i=d\cap U_i\); observe that since \(b,c\) and \(d\) are flat in \(X\), and \(b\cap \overline{U_i}\) is flat in \(\overline{U_i}\), we have \(c_i,d_i\) are flat in \(\overline{U_i}\). Inside each \(U_i\), use Corollary \ref{Filled-homeomorphism} to find a self-homeomorphism \(\varphi_i\) of \(\overline{U_i}\) which fixes the boundary and maps \(d_i\) to \(c_i\). In the end, let \(\varphi\) be equal to \(\varphi_i\) on \(\overline{U_i}\) and the identity elsewhere, and observe that \(\varphi\) is a homeomorphism of \(X\) such that \(\varphi\in\pstab {a}\). By construction, \(\varphi(d)=c\), and since \(\varphi\in\pstab {a}\subseteq \stab D\), we conclude \(c\in D\).
\end{proof}

We note that by Theorem \ref{Compact Non D-Sig C}, it is not true that for every manifold, the ideal of compact nowhere dense sets supports a cone measure. However, there are certainly manifolds where the ideal of compact nowhere dense sets do support a cone measure, as illustrated in the following theorem:
\begin{theorem}\label{Compact NWD Cone}
    Let \(n\in \omega\), and let \(X=\mathbb{R}^n\) with the usual topology. For ease of notation, we will represent \(\mathbb{R}^n\) as \( S^{n-1}\times\mathbb{R}_{\geq 0}\), and define the closed balls \(B_r=S^{n-1}\times [0,r]\). Let \(I\) denote the ideal generated by compact nowhere dense sets. Let \(C\subseteq I\) be the set of flat S-carpets in \(B_r\) for some \(r>0\). \(C\) is a cofinal subset of \(I\) and is the support of a cone measure.
\end{theorem}

\begin{proof}
    That \(C\) is cofinal follows quickly from Theorem \ref{superset-Carpet} and the Heine-Borel theorem. To see that \(C\) is a support of the cone measure, work in the permutation model, let \(D\subseteq C\) be cofinal, and let \(a\in I\) be such that \(\pstab a\subseteq \stab D\). Let \(r>0\) be big enough such that \(a\subseteq B_r\), and find a flat S-carpet \(b\) in \(B_{2r}\) that contains \(a\) and whose intersection with each component \(U\) of \(B_r\setminus a\) is a flat S-carpet in \(\overline{U}\). 
    Now to show that \(\cone b\subseteq D\), let \(c\in C,d\in D\) such that \(b\subseteq c,d\). Let \(s,t\) be such that \(c,d\) are flat S-carpets in \(B_s,B_t\) respectively; observe \(s,t\geq 2r\). Let \(\varphi':\mathbb{R}_{\geq 0}\rightarrow\mathbb{R}_{\geq 0}\) be a homeomorphism which fixes the interval \([0,r]\) and is such that \(\varphi'(t)=s\). Define \(\varphi=id_{S^{n-1}}\times \varphi'\), and observe \(\varphi\) fixes \(B_r\) pointwise and hence fixes \(a'\) pointwise. Now let \(A=\overline{B_s\setminus B_r}\) observe that \(c\cap A\) and \(d\cap A\) are both flat S-carpets in \(A\), and Corollary \ref{Filled-homeomorphism} yields a homeomorphism \(\psi_0\) of \(A\) which fixes the boundary and maps \(\varphi(d)\cap A\) to \(c\cap A\). Additionally, we can proceed as in the proof of Theorem \ref{Carpet-cone} to find a self-homeomorphism \(\psi_1\) of \(B_r\) which fixes \(a\) and the boundary pointwise and maps \(d\cap B_r\) to \(c\cap B_r\). Now let \(\psi\) be the self-homeomorphism of \(\mathbb{R}^n\) obtained by taking the union of \(\psi_0,\psi_1\), and the identity on \(\mathbb{R}^n\setminus B_s\). The composition \(\psi\varphi\) is a self-homeomorphism which fixes \(a\) pointwise and maps \(d\) to \(c\). Since \(\pstab{a}\subseteq\stab{D}\), we conclude \(c\in D\), and so the result follows.
\end{proof}

We now consider results concerning totally disconnected sets. We first note the following theorem presented in \cite{Knaster}:
\begin{theorem}
    If \(C_1,C_2\) are homeomorphic to the Cantor set, \(b_1,b_2\) are nowhere dense subsets of \(C_1,C_2\) respectively, and \(h:b_1\rightarrow b_2\) is a homeomorphism, then \(h\) can be extended to a homeomorphism \(\overline{h}:C_1\rightarrow C_2\).
\end{theorem}

From this theorem, we quickly get the following result:
\begin{theorem}\label{2w-cone}
    Let \(X=2^\omega\) denote the Cantor space with its standard topology. Let \(I\) be the ideal on \(X\) generated by closed nowhere dense sets. Let \(C\subset I\) be the set of perfect nowhere dense sets. \(C\) is a cofinal subset of \(I\) and is the support of a cone measure.
\end{theorem}
\begin{proof}
    Working in the permutation model, let \(D\subseteq C\) be a cofinal subset, and let \(a\in I\) be such that \(\pstab a\subseteq \stab D\). Let \(b\in C\) be a set such that \(a\subseteq b\) is nowhere dense; we claim that \(\cone b\subseteq D\): let \(b\subseteq c\in C\) and \(b\subseteq d\in D\). Observe that \(a\) is nowhere dense in \(c\) and \(d\), and that \(c\) and \(d\) are homeomorphic to the Cantor space. Hence \cite{Knaster} can be used to extend the identity map on \(a\) to a homeomorphism \(h:d\rightarrow c\). Apply \cite{Knaster} again to obtain a homeomorphism \(\overline{h}:2^\omega\rightarrow 2^\omega\), and observe that \(\overline{h}\in\pstab a\subseteq \stab D\), so \(c=h(d)\in D\).
\end{proof}

We next show that a similar result holds when looking at compact totally disconnect subsets of the Euclidean plane. The proof will be very similar, but we will need a different extension theorem given in \cite{Moise}:\begin{theorem}
    If \(M, M'\) are compact totally disconnected sets in \(\mathbb{R}^2\), and \(f:M\rightarrow M'\) is a homeomorphism, then \(f\) extends to a homeomorphism \(\bar{f}:\mathbb{R}^2\rightarrow\mathbb{R}^2\).
\end{theorem}

\begin{theorem}\label{Cone Tot-Disc}
    Let \(X=\mathbb{R}^2\) be the 2-dimensional Euclidean space with standard topology, and let \(I\) be the ideal on \(X\) generated by compact totally disconnected sets. Let \(C\subseteq I\) be the set of perfect compact totally disconnected sets. \(C\) is a cofinal subset of \(I\) and is the support of a cone measure.
\end{theorem}
\begin{proof}
Working in the permutation model, let \(D\subseteq C\) be a cofinal subset, and let \(a\in I\) be such that \(\pstab a \subseteq \stab D\). Let \(b\in C\) be a set such that \(a\subseteq b\) is nowhere dense; we claim that \(\cone b \subseteq D\): let \(b\subseteq c\in C\) and \(b\subseteq d\in D\). Observe that \(a\) is a nowhere dense subset of \(c\) and \(d\), and use the theorem from \cite{Knaster} to extend the identity on \(a\) to a homeomorphism \(\varphi:d\rightarrow c\) between \(d\) and \(c\). From here, extend \(\varphi\) to a homeomorphism of the entire plane \(\overline{\varphi}:\mathbb{R}^2\rightarrow\mathbb{R}^2\) \cite{Moise}. Since \(\overline{\varphi}\in\pstab a,\) we know \(\overline{\varphi}\in\stab D\), and hence \(c=\varphi(d)\in D\).
\end{proof}

We note that in \(\mathbb{R}^3\), the extension theorem given in \cite{Moise} does not hold since there are multiple distinct embeddings of \(2^\omega\) in \(\mathbb{R}^3\). Hence the above argument does not immediately generalize to higher dimensions. It is an open problem to either show that the ideal of compact totally disconnected sets supports a cone measure in \(\mathbb{R}^3\) or to find a specific counterexample witnessing that it does not.

\printbibliography
\end{document}